\documentclass[11pt,letterpaper]{amsart}
\usepackage{amsmath,amssymb, amsthm}
\usepackage{appendix}
\usepackage{amsmath}
\usepackage{amssymb}
\usepackage{amscd}
\usepackage{latexsym}
\usepackage[dvipdfmx]{graphicx}
\usepackage{graphicx}
\usepackage{here}
\usepackage{xcolor}
\usepackage{palatino}
\usepackage{tikz}

\usepackage[colorlinks]{hyperref}
\hypersetup{citecolor=blue}

\newtheorem{theorem}{Theorem}[section]
\newtheorem{lemma}[theorem]{Lemma}

\newtheorem{cor}[theorem]{Corollary}

\newtheorem{question}[theorem]{Question}

\theoremstyle{definition}
\newtheorem{definition}[theorem]{Definition}

\theoremstyle{remark}
\newtheorem{remark}[theorem]{Remark}

\numberwithin{equation}{section}

\newcommand{\R}{{\mathbb R}}

\newcommand{\C}{{\mathbb C}}
\newcommand{\Z}{{\mathbb Z}}
\newcommand{\Q}{{\mathbb Q}}

\newcommand{\rank}{\mathrm{rank}}
\newcommand{\corank}{\mathrm{corank}}
\newcommand{\im}{\mathrm{im}}
\newcommand{\coker}{\mathrm{coker}}
\title[Infinitely many symplectic fillings]
{
{\large Every cusp singularity link admits infinitely many strong symplectic fillings}
}
\author[N Kasuya]{Naohiko KASUYA}
\address{Department of Mathematics, Faculty of Science, Hokkaido University, 
North 10, West 8, Kita-ku, Sapporo, Hokkaido 060-0810, Japan.}
\email{nkasuya@math.sci.hokudai.ac.jp}
\author[T Oba]{Takahiro OBA}
\address{Department of Mathematics, The University of Osaka, 1-1 Machikaneyama, Toyonaka, Osaka 560-0043, Japan}
\email{taka.oba@math.sci.osaka-u.ac.jp}

\begin{document}

\maketitle

\begin{abstract}
In this paper, we show that if the link of an isolated complex surface singularity is either a $Sol^3$-manifold or an $\widetilde{SL}(2;\R)$-manifold with its canonical contact structure, then it admits infinitely many strong symplectic fillings that are pairwise non-diffeomorphic and not related by a sequence of blow-ups or blow-downs. 
As a consequence, the link of any cusp singularity, exceptional unimodal singularity, or hyperbolic Brieskorn singularity admits infinitely many pairwise non-diffeomorphic minimal strong symplectic fillings. 
\end{abstract}
\section{\large Introduction}

One of the principal sources of contact $3$-manifolds is the links of isolated hypersurface singularities in three complex variables. 
These singularities have associated Milnor fibers and minimal resolutions, which provide canonical \textit{strong symplectic fillings} of the links. 
Here, by a strong symplectic filling of a closed positive contact $3$-manifold $(M,\xi)$, we mean a compact connected symplectic $4$-manifold $(X, \omega)$ with boundary such that $\partial X= M$ as oriented manifolds and there exists a Liouville vector field $V$ defined near the boundary pointing outward along $\partial X$ which satisfies $\xi=\ker(\iota_{V}\omega)|_{T\partial X}$.

Based on the modality $m$, which measures the complexity of singularities, Arnol'd classified singularities up to modality $2$; see \cite{Arnold-Gusein-Zade-Varchenko}. 
Singularities with $m=0,1$ are called \textit{simple} and \textit{unimodal}, respectively. 
Unimodal singularities fall into three types: 
\textit{simple elliptic} ($\tilde{E}_6, \tilde{E}_7, \tilde{E}_8$), 
\textit{cusp} ($T_{pqr}$), 
and the $14$ exceptional singularities 
($E_{12}$, $E_{13}$, $E_{14}$, $Z_{11}$, $Z_{12}$, $Z_{13}$, $Q_{10}$, $Q_{11}$, $Q_{12}$, $W_{12}$, $W_{13}$, $S_{11}$, $S_{12}$, $U_{12}$). 
The geometry of the links of singularities with $m \leq 1$ is well-studied. 
Neumann \cite{Neumann} showed that the links of simple, simple elliptic, cusp, and the $14$ exceptional singularities are diffeomorphic to quotients of simply connected Lie groups $SU(2)$, $Nil^3$, $Sol^3$, and $\widetilde{SL}(2;\R)$, respectively, by cocompact lattices. 
Moreover, the canonical contact structures on these links agree with those induced by left-invariant contact structures on these simply connected Lie groups \cite{Ehlers-Neumann-Scherk}. 

\begin{table}[htbp]
\centering
\renewcommand{\arraystretch}{1.6}
\begin{tabular}{c|c}
\hline
Symbol & Normal form   \\
\hline
$\tilde{E}_6$ & $x^3+y^3+z^3+axyz, \; a^3+27\ne 0$ \\
$\tilde{E}_7$ & $x^2+y^4+z^4+axyz, \; a^4-64\ne 0$ \\
$\tilde{E}_8$ & $x^2+y^3+z^6+axyz, \; a^6-432\ne 0$ \\
\hline
$T_{pqr}$ & $x^p+y^q+z^r+axyz, \; a\ne 0, \; \frac{1}{p}+\frac{1}{q}+\frac{1}{r}<1$ \\
\hline
$E_{12}$ & $x^2+y^3+z^7+ayz^5$ \\
$E_{13}$ & $x^2+y^3+yz^5+az^8$ \\
$E_{14}$ & $x^2+xz^4+y^3+axz^6$ \\
$Z_{11}$ & $x^2+y^5+yz^3+ay^4z$ \\
$Z_{12}$ & $x^2+zy^3+yz^4+ay^2z^3$ \\
$Z_{13}$ & $x^2+xy^3+yz^3+ay^5z$ \\
$Q_{10}$ & $x^3+y^4+yz^2+axy^3$ \\
$Q_{11}$ & $x^3+xy^3+yz^2+ay^5$ \\
$Q_{12}$ & $x^3+zy^2+yz^3+axz^4$ \\
$W_{12}$ & $x^2+xy^2+z^5+ay^2z^3$\\
$W_{13}$ & $x^2+xy^2+yz^4+az^6$ \\
$S_{11}$ & $x^4+xy^2+yz^2+ax^3y$ \\
$S_{12}$ & $x^2y+y^2z+z^3x+az^5$ \\
$U_{12}$ & $x^4+zy^2+yz^2+ax^2(y^2+yz+z^2)$ \\
\hline
\end{tabular}
\medskip
\caption{Unimodal singularities}
\label{unimodal}
\end{table}

Ohta and Ono \cite{Ohta-Ono_simple} proved that any minimal strong symplectic filling of the link of a simple singularity is diffeomorphic to its Milnor fiber. 
They \cite{Ohta-Ono03} also showed that any minimal strong symplectic filling of the link of a simple elliptic singularity is diffeomorphic to either its Milnor fiber or its minimal resolution. 
This motivates the study of the corresponding question for the links of cusp singularities and of the $14$ exceptional singularities.

The goal of this paper is to prove that the links of all cusp singularities and the $14$ exceptional unimodal singularities admit 
infinitely many pairwise non-diffeomorphic minimal strong symplectic fillings. 
The outline of the proof is as follows. 
We first prove that any closed contact $3$-manifold admits infinitely many symplectic caps with pairwise distinct $b_2^+$. 
A related statement was claimed by Etnyre and Honda \cite[Theorem~1.3]{Etnyre-Honda}. 
However, there is a gap in their proof. 
We fill this gap in our proof (see Remark~\ref{rem_gap} for the gap in their proof and the difference between the two proofs). 
Now we appeal to a result due independently to Geiges \cite{Geiges} and Mitsumatsu \cite{Mitsumatsu} to carry out the construction. 
For an arbitrary $Sol^3$- or $\widetilde{SL}(2;\R)$-manifold $M$ equipped with the contact structure $\xi$ induced by a left-invariant contact structure on the Lie group, they constructed a $4$-dimensional Liouville domain, with disconnected convex boundary, which is diffeomorphic to $[-1,1] \times M$ and one of whose boundary components is contactomorphic to $(M, \xi)$. 
Hence, for the given link $(M,\xi)$ of a cusp or one of the exceptional $14$ singularities, we take a Liouville domain constructed by Geiges and Mitsumatsu with one boundary component contactomorphic to $(M, \xi)$ and cap off the other boundary component by its symplectic cap. 
This yields a strong symplectic filling of $(M,\xi)$. 
As we have infinitely many symplectic caps of any contact $3$-manifold with pairwise distinct $b_2^+$, we obtain the desired family of minimal strong symplectic fillings of the link $(M,\xi)$. 
Similarly, since the link of a hyperbolic Brieskorn singularity is an $\widetilde{SL}(2;\R)$-manifold, it admits infinitely many minimal strong symplectic fillings.

We would like to point out a couple of related results. 
First, Ohta and Ono \cite{Ohta-Ono} constructed an infinite family of minimal strong symplectic fillings of the link of an isolated hypersurface singularity by using a Liouville domain with disconnected convex boundary. 
Moreover, with the help of such Liouville domains, Geiges and Zehmisch \cite{Geiges-Zehmisch} demonstrated that strong symplectic fillings of the unit cotangent bundle of a Riemann surface of genus at least $2$ can have highly complicated topology. 
They proved that there exists a minimal strong symplectic filling of the unit cotangent bundle whose fundamental group is isomorphic to an arbitrary prescribed finitely presented group.

The strong symplectic fillings constructed in this paper contain symplectic tori; hence they are non-exact and, in particular, are not Stein fillings.
In their recent preprint \cite{Baykur-Nemethi-Plamenevskaya}, Baykur, N\'{e}methi and Plamenevskaya showed that for any positive integer $N$, there exists a cusp singularity whose link admits at least $N$ Stein fillings with pairwise different $b_2$. 
This leads to the following question. 

\begin{question}
Among the cusp singularities and the $14$ exceptional unimodal singularities, is there a singularity whose link admits infinitely many Stein fillings up to homotopy equivalent?
\end{question}

This paper is organized as follows. 
In Section~\ref{singularity links}, we review simply connected $3$-dimensional Lie groups whose quotients by cocompact lattices are diffeomorphic to singularity links, and describe left-invariant contact structures on these Lie groups.
We also specify which singularity corresponds to each Lie group. 
Based on results of this section, in Section~\ref{Liouville domain}, we recall the construction of Liouville domains with disconnected convex boundary.
In Section~\ref{section: main thm}, we prove the main theorem, Theorem~\ref{general}. 
We begin by showing that every contact $3$-manifold admits infinitely many symplectic caps with pairwise distinct $b_2^+$.
We then apply this result to prove Theorem~\ref{general}.

\section{\large Surface singularity links and $3$-dimensional Lie groups}~\label{singularity links}
Let $G$ be a simply-connected $3$-dimensional Lie group, and $\Gamma $ its lattice, i.e., a discrete subgroup. 
A lattice $\Gamma $ is said to be cocompact if the quotient $\Gamma \backslash G$ is compact. 
According to \cite{Raymond-Vasquez}, if $\Gamma \backslash G$ is a closed $3$-manifold, 
then $G$ must be one of the following six Lie groups; 
$$SU(2), \; Nil^3, \; Sol^3, \; \widetilde{SL}(2;\R), \; \widetilde{E}^+(2), \; \R^3, $$
where $Nil^3, Sol^3, \widetilde{SL}(2;\R)$ and $\widetilde{E}^+(2)$ are  
the Heisenberg group, the inhomogeneous Lorentz group $E(1,1)$ and 
the universal covering of $PSL(2;\R)$, 
the universal covering of the affine group $E^+(2)$ consisting of orientation preserving isometeries of $\R^2$, respectively. 

Moreover, it is known that the compact quotient $\Gamma \backslash G$ can be diffeomorphic to 
a singularity link only when the Lie group $G$ is one of the four Lie groups 
$SU(2)$, $Nil^3$, $Sol^3$ and $\widetilde{SL}(2;\R)$ (\cite{Neumann}, see also \cite{Seade}). 
In each case, the canonical contact structure on the singularity link 
is induced by a left-invariant contact structure on the Lie group (\cite{Ehlers-Neumann-Scherk}). 
In the following, we describe the structure of the Lie group $G$ and its left-invariant contact structures in each of the four cases, 
and clarify which singularity links can be realized in the form $\Gamma \backslash G$. 

\subsection{$SU(2)$-case} 
Take a basis $\{e_1, e_2, e_3\}$ of the Lie algebra $su(2)$ by 
$$e_1=\frac{1}{2}\begin{pmatrix} -i&0 \\ 0&i \end{pmatrix}, 
e_2=\frac{1}{2}\begin{pmatrix} 0&-1 \\ 1&0 \end{pmatrix}, 
e_3=\frac{1}{2}\begin{pmatrix} 0&-i\\-i&0 \end{pmatrix}.$$ 
Then we have $[e_1,e_2]=-e_3, [e_2, e_3]=-e_1, [e_3, e_1]=-e_2$ and hence, 
for the dual basis $\{e_1^{\ast }, e_2^{\ast}, e_3^{\ast}\}$, 
$$de_1^{\ast}=e_2^{\ast}\wedge e_3^{\ast}, \; de_2^{\ast }=e_3^{\ast}\wedge e_1^{\ast}, \; de_3^{\ast}=e_1^{\ast}\wedge e_2^{\ast}.$$
Therefore, for any $(a,b,c)\in \R^3\setminus \{ {\bf 0} \}$, 
$ae_1^{\ast }+be_2^{\ast}+ce_3^{\ast}$ defines a positive left-invariant contact form, 
whose kernel is contactomorphic to the standard contact structure $\xi_0$ on $S^3$. 

It is well known that the finite subgroups of $SU(2)$ are precisely the finite cyclic groups $C_r$ of order $r$ and the binary triangle groups $\Gamma(2,2,r)$, $\Gamma(2,3,3)$, $\Gamma(2,3,4)$ and $\Gamma(2,3,5)$. On the other hand, $SU(2)$ acts freely on $\C^2$ except at the only fixed point ${\bf 0}$. Therefore, for each finite subgroup $\Gamma\subset SU(2)$, one obtains an isolated surface singularity $(\Gamma\backslash\C^2,{\bf 0})$. Such a singularity is called a {\it simple singularity}. 
Klein studied simple singularities from the viewpoint of polyhedral groups and invariant polynomials, and showed that they can be realized as the following hypersurface singularities in $\C^3$:
\begin{eqnarray*}
A_n: x^2+y^2+z^{n+1}=0 \;\; &\text{if}&  \Gamma =C_{n+1} \; (n\geq 1),\\
D_n: x^2+y^2z+z^{n-1}=0  &\text{if}&  \Gamma =\Gamma(2,2,n-2) \; (n\geq 4),\\
E_6: x^2+y^3+z^4=0 \;\;\;\;\;\; &\text{if}&  \Gamma =\Gamma (2,3,3),\\
E_7: x^2+y^3+yz^3=0 \;\;\;\; &\text{if}&  \Gamma =\Gamma(2,3,4),\\
E_8: x^2+y^3+z^5=0 \;\;\;\;\;\; &\text{if}&  \Gamma =\Gamma(2,3,5).
\end{eqnarray*}
The singularity link is diffeomorphic to the compact quotient $\Gamma \backslash SU(2)$. 
Since the standard contact structure $\xi _0$ on $SU(2)\cong S^3$ is left-invariant, 
it descends to a contact structure on $\Gamma \backslash SU(2)$,  
which coincides with the canonical contact structure on the singularity link. 

\subsection{$Nil^3$-case}
The Lie group $Nil^3$ is the Heisenberg group, i.e., 
the group of $3\times 3$ upper triangular matrices of the form $\begin{pmatrix}1&x&z\\0&1&y\\0&0&1\end{pmatrix}$. 
Take a basis $\{e_1, e_2, e_3\}$ of the Lie algebra $nil^3$ by 
$$
e_1=\begin{pmatrix} 0&0&0 \\ 0&0&1 \\ 0&0&0 \end{pmatrix}, 
e_2=\begin{pmatrix} 0&1&0 \\ 0&0&0 \\ 0&0&0 \end{pmatrix}, 
e_3=\begin{pmatrix} 0&0&1 \\ 0&0&0 \\ 0&0&0 \end{pmatrix}.$$ 
Then, we have $[e_1, e_2]=-e_3, [e_2, e_3]=0, [e_3, e_1]=0$. 
Hence, for the dual basis $\{e_1^{\ast }, e_2^{\ast }, e_3^{\ast }\}$, 
$$de_1^{\ast}=0, \; de_2^{\ast}=0, \; de_3^{\ast}=e_1^{\ast}\wedge e_2^{\ast}.$$
Therefore, $e_3^{\ast}$ defines a positive left-invariant contact form on $Nil^3$. 

For a cocompact lattice $\Gamma $, 
the compact quotient $\Gamma \backslash Nil^3$ is called a $Nil^3$-manifold. 
It is diffeomorphic to the total space of an orientable $S^1$-bundle over $T^2$, 
which also fibers over $S^1$ as a parabolic $T^2$-bundle. 

\begin{definition}[simple elliptic singularity]
A normal surface singularity is called a {\it simple elliptic singularity} if the exceptional set of its minimal resolution is a non-singular elliptic curve of negative self-intersection number $-k$. The number $k>0$ is called the {\it degree} of the simple elliptic singularity. 
\end{definition}
It is clear that the link of a simple elliptic singularity is a {\it $Nil^3$-manifold}. 
Conversely, a normal surface singularity is simple elliptic if its link is a $Nil^3$-manifold \cite{Neumann2}. 
The simple elliptic singularities of degrees $3,2$ and $1$ are precisely the hypersurface singularities $\tilde{E}_6, \tilde{E}_7$ and $\tilde{E}_8$,  respectively, which appear in the list of unimodal singularities \cite{Saito}. 
In each case, the canonical contact structure of the link coincides with 
the positive contact structure on $\Gamma \backslash Nil^3$ induced by $e_3^{\ast}$.

\subsection{$Sol^3$-case}
The Lie group $Sol^3$ is the split extension of $\R^2$ by $\R$ whose group structure is given by 
$$(x,y; z)\cdot (x',y'; z')=(x+e^zx', y+e^{-z}y'; z+z') \;\; \text{on} \;\;  \R^2\times \R.$$
Take a basis $\{e_1, e_2, e_3\}$ of the Lie algebra $sol^3$ by 
$$e_1=e^z\partial _x-e^{-z}\partial_y, \; e_2=e^z\partial _x+e^{-z}\partial_y, \; e_3=\partial _z.$$ 
Then we have $[e_1,e_2]=0, [e_2, e_3]=-e_1, [e_3, e_1]=e_2$. 
Hence, for the dual basis $\{e_1^{\ast }, e_2^{\ast}, e_3^{\ast}\}$, 
$$de_1^{\ast}=e_2^{\ast}\wedge e_3^{\ast}, \; de_2^{\ast }=e_1^{\ast}\wedge e_3^{\ast}, \; de_3^{\ast}=0.$$
Therefore, $e_1^{\ast}$ and $e_2^{\ast}$ are positive and negative contact forms, respectively. 
When $\Gamma $ is a cocompact lattice of $Sol^3$, the compact quotient $\Gamma \backslash Sol^3$ is called a {\it $Sol^3$-manifold}. It is diffeomorphic to the total space of a hyperbolic $T^2$-bundle over $S^1$. 
Let $A\in SL(2;\Z)$ with $\mathrm{tr} (A)>2$ be the monodromy of this $T^2$-bundle, and denote its mapping torus by $T_A$. Then $\Gamma \backslash Sol^3\cong T_A$. Moreover, $T_A$ carries a positive contact structure $\xi _+$ and a negative contact structure $\xi _-$ induced by $e_1^{\ast}$ and $e_2^{\ast}$, respectively.

\begin{definition}[cusp singularity]
A normal surface singularity is called a {\it cusp singularity} if the exceptional set of its minimal resolution is a rational curve with one node or a cycle of rational curves. 
\end{definition}

It is easy to verify that the link of a cusp singularity is a $Sol^3$-manifold. 
Conversely, a normal surface singularity is a cusp singularity if its link is a $Sol^3$-manifold \cite{Neumann2}.  
In fact, cusp singularities are precisely the singularities known as {\it Hilbert modular cusps} \cite{Hirzebruch, Laufer, Laufer2}. 
Those cusp singularities that can be realized as hypersurface singularities are exactly the $T_{pqr}$-singularities appearing in the list of unimodal singularities \cite{Karras}. 
Moreover, the canonical contact structure on the link is contactomorphic to $(T_A, \xi _+)$ \cite{Kasuya}. 

\subsection{$\widetilde{SL}(2;\R)$-case}
Take a basis $\{e_1, e_2, e_3\}$ of the Lie algebra $psl(2;\R)$ by 
$$e_1=\frac{1}{2}\begin{pmatrix} 0&-1 \\ 1&0 \end{pmatrix}, 
e_2=\frac{1}{2}\begin{pmatrix} 1&0 \\ 0&-1 \end{pmatrix}, 
e_3=\frac{1}{2}\begin{pmatrix} 0&1 \\ 1&0 \end{pmatrix}.$$ 
Then we have $[e_2,e_3]=-e_1, [e_3, e_1]=e_2, [e_1, e_2]=e_3$. 
Hence, for the dual basis $\{e_1^{\ast }, e_2^{\ast}, e_3^{\ast}\}$, 
$$de_1^{\ast}=e_2^{\ast}\wedge e_3^{\ast}, \; de_2^{\ast}=e_1^{\ast}\wedge e_3^{\ast}, \; de_3^{\ast }=e_2^{\ast}\wedge e_1^{\ast}.$$
Therefore, $e_1^{\ast}$ is a positive contact form, while $e_2^{\ast}$ and $e_3^{\ast}$ are negative contact forms. 
For a cocompact lattice $\Gamma $, 
the quotient $\Gamma \backslash \widetilde{SL}(2;\R)$ is called an {\it $\widetilde{SL}(2;\R)$-manifold}. 
Compared with the other three cases, isolated surface singularities whose links are $\widetilde{SL}(2;\R)$-manifolds form a broader class.
\begin{definition}[quasihomogeneous singularity]
A normal surface singularity is said to be {\it quasihomogeneous (q.h.)} if it admits a good $\C^{\ast}$-action. 
\end{definition}
The link of a Gorenstein quasihomogeneous singularity is an $\widetilde{SL}(2;\R)$-manifold 
if and only if the singularity is neither simple nor simple elliptic. Such a singularity is said to be {\it hyperbolic quiasihomogeneous}. 
Conversely, any $\widetilde{SL}(2;\R)$-manifold can be realized as the link of a hyperbolic quasihomogeneous singularity \cite{Dolgachev, Neumann}. The canonical contact structure on such a singularity link $\Gamma\backslash \widetilde{SL}(2;\R)$ is induced by $e_1^{\ast }$. 

A hyperbolic quasihomogeneous singularity is called a {\it triangle singularity} if the lattice $\Gamma $ descends to the triangle group $\Sigma (p,q,r)$ $(\frac{1}{p}+\frac{1}{q}+\frac{1}{r}<1)$ in $PSL(2;\R)$. Dolgachev~\cite{Dolgachev2} showed that only $14$ triangle singularities can be realized as hypersurface singularities in $\C^3$. 
They correspond to the $14$ exceptional unimodal singularities with parameter $a=0$. 
Note that when $a\ne 0$, each exceptional unimodal singularity is non-quasihomogeneous. 
However, such a singularity is a deformation of one of the $14$ hypersurface triangle singularities. 
Since such a deformation preserves both the diffeomorphism type of the link and its canonical contact structure, we may, for our purposes, regard any exceptional unimodal singularity as a hypersurface triangle singularity.
Therefore, we shall not discuss non-quasihomogeneous exceptional unimodal singularities separately. 

\begin{table}[htbp]
\centering
\renewcommand{\arraystretch}{1.2}
\begin{tabular}{c|c|c|c}
\hline
Lie group $G$ & Singularity & Quadratic form & Left inv. contact forms  \\
\hline
$SU(2)$ & simple & $a^2+b^2+c^2$ & (+) $e_1^{\ast}$ ($e_2^{\ast}, e_3^{\ast}$) \\
$Nil^3$ & simple elliptic & $c^2$ & (+) $e_3^{\ast}$ \\
$Sol^3$ & cusp & $a^2-b^2$ & (+) $e_1^{\ast}\;$ (-) $e_2^{\ast}$ \\
$\widetilde{SL}(2;\R)$ & hyperbolic q.h. & $a^2-b^2-c^2$ & (+) $e_1^{\ast}\;$ (-) $e_2^{\ast}$ ($e_3^{\ast}$) \\
\hline
\end{tabular}
\medskip
\caption{Lie groups, singularities, quadratic forms and left invariant contact forms. In the second column, a singularity means one whose link is diffeomorphic to $\Gamma \backslash G$.
In the third column, the quadratic forms $q(a,b,c)$ defined in Remark~\ref{quadratic} are listed. 
In the fourth column, the $1$-forms $e_1^{\ast}$, $e_2^{\ast}$, and $e_3^{\ast}$ on $G$ are defined above, and the signs indicate the signs of contact structures.}
\label{Lie contact}
\end{table}

\begin{remark}\label{quadratic}
Notice that in all four cases, it holds that $$e^{\ast }_i\wedge de^{\ast}_j=0 \;\; \text{if} \;\; i\ne j.$$ 
Consequently, a left-invariant $1$-form $ae^{\ast }_1+be^{\ast }_2+ce^{\ast }_3$ defines a positive contact structure (resp. a negative contact structure, a foliation) if and only if the value of the diagonalized quadratic form $q (a, b, c)$ is positive (resp. negative, zero), where 
$$(ae^{\ast }_1+be^{\ast }_2+ce^{\ast }_3)\wedge d(ae^{\ast }_1+be^{\ast }_2+ce^{\ast }_3)
=q(a,b,c) (e^{\ast}_1\wedge e^{\ast}_2\wedge e^{\ast}_3).$$
For the explicit quadratic form in each case, see Table~\ref{Lie contact}. 
\end{remark}

\begin{remark}
Extending Klein's work on simple singularities, Milnor \cite{Milnor} showed that 
a 3-dimensional Brieskorn manifold $M(p,q,r)$, the link of a Brieskorn singularity $x^p+y^q+z^r=0$, 
admits one of the three geometries $SU(2), Nil^3$ or $\widetilde{SL}(2;\R)$ according as $\frac{1}{p}+\frac{1}{q}+\frac{1}{r}-1$ is positive, zero or negative. The works of Dolgachev~\cite{Dolgachev} and Neumann~\cite{Neumann} on quasihomogeneous singularities built upon Milnor's result. The cases corresponding to the geometries $SU(2), Nil^3$ and $\widetilde{SL}(2;\R)$, with the exception of non-quasihomogeneous exceptional unimodal singularities, all fall within this framework. 
Moreover, the canonical contact structure on the singularity link can be identified easily, since the contact form can be deformed so that all Reeb orbits coincide with the $S^1$-orbits generated by the $\C^{\ast}$-action.

On the other hand, cusp singularities are not quasihomogeneous and therefore must be treated separately as Hilbert modular cusps. However, thanks to the foundational work of Hirzebruch \cite{Hirzebruch}, the theory is highly developed, and it is not difficult to understand the geometry of the singularity link, including the contact structure on it. 
\end{remark}

\section{Liouville domains with disconnected convex boundaries}\label{Liouville domain}
The first example of a Liouville domain with disconnected convex boundary was constructed by McDuff~\cite{McDuff}. 
She used the cotangent bundle of a closed hyperbolic surface and its geodesic flow for the construction.  
Later, Geiges~\cite{Geiges} and Mitsumatsu~\cite{Mitsumatsu} independently generalized her construction from the viewpoint of $3$-dimensional Lie groups. \\

\begin{theorem}[Geiges~\cite{Geiges}, Mitsumatsu~\cite{Mitsumatsu}]~\label{GM}
Let $M$ be a $Sol^3$-manifold or an $\widetilde{SL}(2;\R)$-manifold. 
Then the product $[-1,1]\times M$ admits a structure of a Liouville domain 
with disconnected convex boundary $\{-1\}\times M \sqcup \{1\}\times M$. 
\end{theorem}

\begin{proof}
In each of the $Sol^3$ and $\widetilde{SL}(2;\R)$ cases, there exists a pair $(e_1^{\ast}, e_2^{\ast})$ of positive and negative contact forms. 
Let $\alpha _1$ and $\alpha _2$ denote the contact froms on $M$ induced by $e_1^{\ast}$ and $e_2^{\ast}$, respectively. 
We define a $1$-form $\alpha $ on the product $[-1,1]\times M$ by 
$$\alpha =\frac{1}{2}\Big((1+t)\alpha _1+(1-t)\alpha _2\Big),$$
where $t$ denotes the coordinate on the interval $[-1,1]$. 
Then $\alpha $ is a Liouville form, and the associated Liouville vector field is given by $X=t\partial _t-v_3$, 
where $v_3$ is a vector field on $M$ induced by $e_3$. 
The vector field $X$ is outward transverse to the boundary $\partial ([-1,1]\times M)$. 
Hence, $[-1,1]\times M$ is a Liouville domain and its boundary is convex. 
\end{proof}

As Geiges \cite{Geiges} pointed out, the key point of the above construction is the existence on  
$Sol^3$ and $\widetilde{SL}(2;\R)$ of a pair $(e_1^{\ast}, e_2^{\ast})$ of positive and negative contact forms satisfying 
$$e^{\ast}_1\wedge de^{\ast}_2=e^{\ast}_2\wedge de^{\ast}_1=0.$$ 
Mitsumatsu \cite{Mitsumatsu} explained the geometric meaning of this condition from the viewpoint of so-called {\it algebraic Anosov flows}, 
and called the associated pair $(\xi _+=\ker \alpha _1, \xi_-=\ker \alpha _2)$ of positive and negative contact structures a {\it bi-contact structure}. 

\section{\large Main Theorem}\label{section: main thm}

In this section, we will prove the main theorem in this paper. 
We begin with a result about the topology of symplectic caps of contact $3$-manifolds. 
Here, 
by a \textit{symplectic cap} of a closed positive contact $3$-manifold $(M, \xi)$, we mean 
a compact symplectic $4$-manifold $(X, \omega)$ with boundary such that $\partial X= -M$ as oriented manifolds and there exists a Liouville vector field $V$ defined near the boundary pointing inward along $\partial X$ which satisfies $\xi =\ker (\iota _V\omega )|_{T\partial X}$.

\begin{theorem}\label{caps}
Any closed positive contact $3$-manifold admits infinitely many symplectic caps whose $b_2^+$-invariants are pairwise distinct. 
\end{theorem}

\begin{proof}
Let $(M, \xi )$ be any closed positive contact $3$-manifold. Taking its symplectization, we obtain a symplectic cobordism $Y$ from $(M, \xi )$ to itself which is diffeomorphic to the product $[-1,1]\times M$. 

Now we take a Darboux ball $U$ in the convex side $\{1\}\times M$. 
Inside $U$, choose a right-handed Legendrian trefoil $L$ with maximal Thurston--Bennequin number and a Legendrian unknot $L'$ linking with $L$ once.
According to \cite{Ekholm-Honda-Kalman}, the Legendrian knot $L$ admits a genus-one Lagrangian filling $S$ in $[-1,1]\times U$ so that $S\cap ([0,1]\times U)=[0,1]\times L$. 

Attaching Weinstein $2$-handles $H$ and $H'$ to $Y$ along $L$ and $L' $ in $\{1\}\times M$, respectively, we obtain an exact symplectic cobordism
\[
C=Y\cup_L H \cup_{L'} H'.
\] 
Let $D$ be the core Lagrangian disk of the Weinstein $2$-handle $H$. 
Then, $S$ and $D$ are naturally glued together along $L$ to form a Lagrangian torus $T$ inside the cobordism $C$. 
Notice that $T$ represents a nontrivial torsion-free class in the relative homology group $H_2(C, \partial C)$ since $[T] \in H_2(C, \partial C)$ has a nontrivial intersection with the homology class represented by a sphere arising from the attachment of $H'$.
Hence we can apply Gompf's theorem \cite[Lemma~1.6]{Gompf}. Namely, we can make $T$ a symplectic torus with trivial normal bundle by perturbing the symplectic form only near the torus $T$. 
By a slight abuse of notation, we still denote the deformed symplectic cobordism by $C$. 
By Gay's result \cite[Theorem~1.1]{Gay}, the convex side $\partial _+C$ of $C$ can be capped by some symplectic cap $V$. 
Thus we obtain a symplectic cap $X:=C\cup V$ of $(M, \xi )$. 
Let $\pi \colon E(n)\to \C P^1$ be the elliptic fibration obtained as the symplectic sum of $n$ copies of rational elliptic surface $E(1)$. Taking the symplectic sum of $X$ and the elliptic surface $E(n)$ along $T$ and a regular fiber of $\pi $, we obtain a new symplectic cap  $X_n$ of $(M, \xi )$. 
As will be shown in Lemma~\ref{estimate 2}, the absolute value
$$
|b_2^+(X_n)-b_2^+(X)-2n|
$$ 
is bounded above by a constant independent of $n$.
Thus $\{b_2^+(X_n) \mid n\in \Z_{>0} \}$ is an infinite set. 
This completes the proof. 
\end{proof}

\begin{remark}\label{rem_gap}
Etnyre and Honda~\cite[Theorem~1.3]{Etnyre-Honda} proved that any closed positive contact $3$-manifold admits a symplectic cap (this was also reproven by Gay \cite{Gay} and Ding--Geiges \cite{Ding-Geiges}). 
In the same theorem, they further claimed that any closed positive contact $3$-manifold admits infinitely many symplectic caps that are pairwise non-diffeomorphic and are not related by a sequence of blow-ups and blow-downs. 
However, there is a gap in their proof of the latter claim. 
Specifically, they asserted that the Stein cobordism they construct contains a symplectically embedded torus, 
which is impossible since an exact symplectic manifold does not contain any closed symplectic submanifold. 
The above proof of Theorem~\ref{caps} fills this gap by constructing a Lagrangian torus $T$ and applying Gompf's theorem. 

Our construction of symplectic caps also differs from that of Etnyre and Honda. 
In \cite{Etnyre-Honda}, they first construct a Stein cobordism from a given contact manifold to a \textit{Stein fillable} contact manifold. 
They then appeal to a result of Lisca--Mati\'{c} \cite{Lisca-Matic} to cap off this Stein fillable contact boundary. 
In our construction, by contrast, the contact manifold at the positive end of the cobordism $C$ in the above does not need to be Stein fillable; 
instead, we cap it off using a result of Gay~\cite{Gay}.
We also note that our Theorem~\ref{caps} improves upon~\cite[Theorem~1.3]{Etnyre-Honda} in the sense that it directly estimates the $b_2^+$-invariants of the symplectic caps $X_n$, which is crucial for the proof of our main theorem. 

\end{remark}

In the following, we denote the topological invariants of $X$ and $X_n$ by the notations without and with tildes, respectively. For example, $b_i$ denotes the $i$-th Betti number of $X$ and 
$b_2^+$ (resp. $b_2^-, b_2^0$) denotes the rank of maximal positive-definite (resp. negative-definite, null) subspace with respect to the intersection form on $H_2(X; \Q)$. 
Similarly, $\tilde{b}_i, \tilde{b}_2^+, \tilde{b}_2^-$ and $\tilde{b}_2^0$ denote the corresponding invariants of $X_n$. 
We now estimate the absolute value $|\tilde{b}_2^+-b_2^+-2n|$ from above. 

\begin{lemma}\label{estimate}
The following inequalities hold: 
$$-2\leq \tilde{b}_1-b_1\leq 0,\quad
-b_2(M) \leq \tilde{b}_2^0-b_2^0 \leq b_2(M),
$$
$$
-b_2(M)-2\leq \tilde{b}_3-b_3\leq b_2(M).
$$
\end{lemma}
\begin{proof}
Firsrt we prove that $-2\leq \tilde{b}_1-b_1\leq 0$. 
Let $F$ and $N(F)$ be a regular fiber of $\pi \colon E(n)\to \C P^1$ and its tubular neighborhood, respectively. 
Also we denote a tubular neighborhood of the torus $T\subset X$ by $N(T)$. 
Then, $X_n$ can be topologically described as $$X_n=(X\setminus N(T))\cup (E(n)\setminus N(F)).$$ 
Since $E(n)\setminus N(F)$ is simply-connected, we obtain the following exact sequence as a part of the Mayer--Vietoris sequence: $$H_1(\partial N(T))\to H_1(X\setminus N(T))\to H_1(X_n)\to 0.$$
Hence, $b_1(X_n)$ coincides with the corank of $$i_{\ast }\colon H _1(\partial N(T))\to H_1(X\setminus N(T)), $$
where $i\colon \partial N(T) \to X\setminus N(T)$ is the inclusion map. 
Now we consider the following commutative diagrams of two homology exact sequences, 
where all the vertical homomorphisms are induced by the natural inclusion maps $\iota _k$ ($k=1,2,3$): 
\[
\begin{CD}
H_1(\partial N(T))@>{i_{\ast }}>> H_1(X\setminus N(T))@>{j_{\ast}}>> H_1(X\setminus N(T), \partial N(T)) \\
@V{(\iota_1)_{\ast}}VV  @V{(\iota_2)_{\ast}}VV @V{(\iota_3)_{\ast}}VV \\
H_1(N(T))@>{{i'}_{\ast}}>> H_1(X) @>{{j'}_{\ast}}>> H_1(X, N(T))
\end{CD}
\]
We want to show that the cokernel of $i_{\ast }$ is isomorphic to that of 
$${i'} _{\ast }\colon \Z^2=H_1(N(T))\to H_1(X),$$ where ${i'} \colon N(T)\to X$ is the inclusion. 
In order for that, it is enough to prove that the image of $j_{\ast }$ is isomorphic to that of ${j'}_{\ast }$. 
Here we notice that the homomorphism $$(\iota _2)_{\ast } \colon H_1(X\setminus N(T))\to H_1(X)$$ is surjective, 
since we have $$H_1(X, X\setminus N(T))\cong H_1(N(T), \partial N(T))\cong H_1(T^2\times D^2, T^3)= 0. $$
In addition, the homomorphism $$(\iota _3)_{\ast } \colon H_1(X\setminus N(T), \partial N(T))\to H_1(X, N(T))$$ 
is an excision isomorphism. 
Hence, we see that $$\im (j_{\ast })\cong \im ((\iota_3)_{\ast}\circ {j_{\ast}})= \im ({j'}_{\ast}\circ (\iota _2)_{\ast})=\im ({j'}_{\ast}).$$ 
Thus we have shown that $\coker ({i} _{\ast })\cong \coker ({i'} _{\ast })$. 
Therefore, we obtain $b_1(X)-2\leq b_1(X_n)=\corank ({i'} _{\ast }) \leq b_1(X)$, and hence, $-2\leq \tilde{b}_1-b_1\leq 0$. 
\\

\noindent
Next we show that $-b_2(M) \leq \tilde{b}_2^0-b_2^0 \leq b_2(M)$. 
Let $j\colon \partial X\to X$ be the inclusion. 
Since $b_2^0=\mathrm{rank} (j_{\ast }(H_2(\partial X;\Q)))$ and $\partial X=M$, we have $0\leq b_2^0 \leq b_2(M)$. 
Smilarly, $0\leq \tilde{b}_2^0 \leq b_2(M)$. Then it follows that $$-b_2(M) \leq \tilde{b}_2^0-b_2^0 \leq b_2(M).$$

\noindent
Finally, we show that $-b_2(M)-2\leq \tilde{b}_3-b_3\leq b_2(M)$. 
By the Poincar\'{e} duality, we have that $b_3=\rank (H_1(X, \partial X))$ and $\tilde{b}_3=\rank (H_1(X_n, \partial X_n))$.  
By the homology exact sequence for the pair $(X, \partial X)$, we obtain the following exact sequence: 
\begin{eqnarray*}
H_1(\partial X)\to H_1(X)\to H_1(X, \partial X)\to 0, 
\end{eqnarray*}
where $j\colon \partial X\to X$ is the inclusion. 
Hence,  
\begin{eqnarray*}
b_1-b_1(\partial X)\leq \rank (H_1(X, \partial X))=\corank (j_{\ast }) \leq b_1. 
\end{eqnarray*}
Similarly, we have $$\tilde{b}_1-b_1(\partial X_n)\leq \rank (H_1(X_n, \partial X_n)) \leq \tilde{b}_1.$$
Recalling that $\partial X=\partial X_n=M$ and $b_1(M)=b_2(M)$, we obtain 
$$b_1-b_2(M)\leq b_3\leq b_1, \;\; \tilde{b}_1-b_2(M) \leq \tilde{b}_3\leq \tilde{b}_1.$$
Therefore, 
$$\tilde{b}_1-b_1-b_2(M)\leq \tilde{b}_3-b_3\leq \tilde{b}_1-b_1+b_2(M).$$
Since $-2\leq \tilde{b}_1-b_1\leq 0$, it follows that $-b_2(M)-2\leq \tilde{b}_3-b_3\leq b_2(M)$. 
\end{proof}

\begin{lemma}\label{estimate 2}
$$-b_2(M)-2\leq \tilde{b}_2^+-b_2^+-2n\leq b_2(M).$$ 
\end{lemma}
\begin{proof}
Wall's signature formula shows that 
$$\sigma ({X}_n)=\sigma (X)+\sigma (E(n)).$$
Since ${X}_n$ is the symplectic sum of $X$ and $E(n)$ along embedded tori, 
it follows that $$\chi (X_n)=\chi (X)+\chi (E(n)).$$ 
Recalling that $\sigma (E(n))=-8n$ and $\chi (E(n))=12n$, 
we have 
\begin{eqnarray*}
\tilde{b}_2^+-\tilde{b}_2^-&=&b_2^+-b_2^--8n, \\
1-\tilde{b}_1+\tilde{b}_2-\tilde{b}_3&=&1-b_1+b_2-b_3+12n.
\end{eqnarray*}
Adding the corresponding sides and simplifying, we obtain 
$$\tilde{b}_2^+-b_2^+-2n=\frac{1}{2}(b_2^0-\tilde{b}_2^0+\tilde{b}_3-b_3+\tilde{b}_1-b_1),$$
where we use $b_2=b_2^0+b_2^++b_2^-$ and $\tilde{b}_2=\tilde{b}_2^0+\tilde{b}_2^++\tilde{b}_2^-$.
Therefore, by Lemma~\ref{estimate}, 
$$-b_2(M)-2\leq \tilde{b}_2^+-b_2^+-2n\leq b_2(M).$$
\end{proof}
 
As a corollary to Theorem~\ref{caps}, we obtain the following assertion, which is our main theorem. 

\begin{theorem}\label{general}
If the link of an isolated complex surface singularity is either a $Sol^3$-manifold or an $\widetilde{SL}(2;\R)$-manifold with its canonical contact structure, then it admits infinitely many strong symplectic fillings that are pairwise non-diffeomorphic 
and not related by a sequence of blow-ups or blow-downs. 
\end{theorem}

\begin{proof}
Let $M$ be either a $Sol^3$-manifold or an $\widetilde{SL}(2;\R)$-manifold, 
and $(\xi_+, \xi _-)$ the bi-contact structure associated with its canonical Anosov flow. 
The convex boundary of the Liouville domain $W$ obtained in Section~\ref{Liouville domain} 
consists of the two connected components $\partial _+W$ and $\partial _-W$ 
which are contactomorphic to $(M, \xi _+)$ and $(-M, \xi_-)$, respectively. 
Gluing any symplectic cap $V$ of $(-M, \xi _-)$ to $W$ by a contactomorphism $\phi \colon (-M, \xi _-)\to \partial _-W$, 
we obtain a strong symplectic filling $\tilde{V}=W\cup _{\phi }V$ of $(M, \xi _+)$ which is diffeomorphic to $V$. 

Using $X_n$ in the proof of Theorem~\ref{caps} as $V$, we see that 
$\tilde{X}_n=W\cup _{\phi }X_n$ is a strong symplectic filling of $(M, \xi _+)$ with $b_2^+(\tilde{X}_n)=b_2^+(X_n)$. 
By Theorem~\ref{caps}, the family $\{ \tilde{X}_n \}_{n>0}$ contains infinite members whose $b_2^+$-invariants are pairwise distinct. This produces infinitely many strong symplectic fillings of $(M, \xi _+)$ 
that are pairwise non-diffeomorphic 
and not related by a sequence of blow-ups or blow-downs. 
\end{proof}

As a consequence, the link of any cusp singularity, exceptional unimodal singularity, or 
hyperbolic Brieskorn singularity admits infinitely many pairwise non-diffeomorphic minimal strong symplectic fillings. 
Recall that the link of a Gorenstein quasihomogeneous singularity is an $\widetilde{SL}(2;\R)$-manifold 
if and only if the singularity is neither simple nor simple elliptic. 
Combining these with the results of Ohta--Ono \cite{Ohta-Ono03, Ohta-Ono_simple}, we obtain the following corollaries. 

\begin{cor}
Among simple and unimodal singularities, 
cusp singularities and the $14$ exceptional unimodal singularities are the only ones 
whose links admit infinitely many pairwise non-diffeomorphic 
minimal strong symplectic fillings. 
\end{cor}

\begin{cor}
The link of a Gorenstein quasihomogeneous surface singularity admits infinitely many pairwise non-diffeomorphic 
minimal strong symplectic fillings if and only if the singularity is neither simple nor simple elliptic. 
In particular, the $3$-dimensional Brieskorn manifold $M(p,q,r)$ admits infinitely many pairwise non-diffeomorphic 
minimal strong symplectic fillings if and only if $\frac{1}{p}+\frac{1}{q}+\frac{1}{r}<1$. 
\end{cor}

\subsection*{Acknowledgements}
This work was supported by the Japan Society for the Promotion of Science (JSPS) KAKENHI Grant Numbers JP21H00985, JP22K13913, JP23H05437, JP24H00182, JP25K06968 and JP26K06780.


\begin{thebibliography}{10}

\bibitem{Arnold-Gusein-Zade-Varchenko}
V.~I. Arnol'd, S.~M. Guse\u{i}n-Zade, and A.~N. Varchenko, \emph{Singularities
  of differentiable maps. {V}ol. {I}}, Monographs in Mathematics, vol.~82,
  Birkh\"auser Boston, Inc., Boston, MA, 1985, The classification of critical
  points, caustics and wave fronts, Translated from the Russian by Ian Porteous
  and Mark Reynolds. \MR{777682}

\bibitem{Baykur-Nemethi-Plamenevskaya}
R.~\.Inan\c~c Baykur, Andr\'as N\'emethi, and Olga Plamenevskaya, \emph{Stein
  fillings vs. milnor fibers}, arXiv preprint arXiv:2403.12216 (2024).

\bibitem{Ding-Geiges}
Fan Ding and Hansj\"org Geiges, \emph{A {L}egendrian surgery presentation of
  contact 3-manifolds}, Math. Proc. Cambridge Philos. Soc. \textbf{136} (2004),
  no.~3, 583--598. \MR{2055048}

\bibitem{Dolgachev2}
I.~V. Dolgachev, \emph{Automorphic forms, and quasihomogeneous singularities},
  Funkcional. Anal. i Prilo\v zen. \textbf{9} (1975), no.~2, 67--68.
  \MR{568895}

\bibitem{Dolgachev}
Igor~V. Dolgachev, \emph{On the link space of a {G}orenstein quasihomogeneous
  surface singularity}, Math. Ann. \textbf{265} (1983), no.~4, 529--540.
  \MR{721886}

\bibitem{Ehlers-Neumann-Scherk}
F.~Ehlers, W.~D. Neumann, and J.~Scherk, \emph{Links of surface singularities
  and {CR} space forms}, Comment. Math. Helv. \textbf{62} (1987), no.~2,
  240--264. \MR{896096}

\bibitem{Ekholm-Honda-Kalman}
Tobias Ekholm, Ko~Honda, and Tam\'{a}s K\'{a}lm\'{a}n, \emph{Legendrian knots
  and exact {L}agrangian cobordisms}, J. Eur. Math. Soc. \textbf{18} (2016),
  no.~11, 2627--2689. \MR{3562353}

\bibitem{Etnyre-Honda}
John~B. Etnyre and Ko~Honda, \emph{On symplectic cobordisms}, Math. Ann.
  \textbf{323} (2002), no.~1, 31--39. \MR{1906906}

\bibitem{Gay}
David~T. Gay, \emph{Explicit concave fillings of contact three-manifolds},
  Math. Proc. Cambridge Philos. Soc. \textbf{133} (2002), no.~3, 431--441.
  \MR{1919715}

\bibitem{Geiges}
Hansj\"org Geiges, \emph{Examples of symplectic {$4$}-manifolds with
  disconnected boundary of contact type}, Bull. London Math. Soc. \textbf{27}
  (1995), no.~3, 278--280. \MR{1328705}

\bibitem{Geiges-Zehmisch}
Hansj{\"o}rg Geiges and Kai Zehmisch, \emph{Symplectic fillings of unit
  cotangent bundles of hyperbolic surfaces}, Israel J. Math. (2025), 1--14.

\bibitem{Gompf}
Robert~E. Gompf, \emph{A new construction of symplectic manifolds}, Ann. of
  Math. (2) \textbf{142} (1995), no.~3, 527--595. \MR{1356781}

\bibitem{Hirzebruch}
Friedrich E.~P. Hirzebruch, \emph{Hilbert modular surfaces}, Enseign. Math. (2)
  \textbf{19} (1973), 183--281. \MR{393045}

\bibitem{Karras}
Ulrich Karras, \emph{Deformations of cusp singularities}, Several complex
  variables ({P}roc. {S}ympos. {P}ure {M}ath., {V}ol. {XXX}, {P}art 1,
  {W}illiams {C}oll., {W}illiamstown, {M}ass., 1975), Proc. Sympos. Pure Math.,
  vol. Vol. XXX, Part 1, Amer. Math. Soc., Providence, RI, 1977, pp.~37--44.
  \MR{472811}

\bibitem{Kasuya}
Naohiko Kasuya, \emph{The canonical contact structure on the link of a cusp
  singularity}, Tokyo J. Math. \textbf{37} (2014), no.~1, 1--20. \MR{3238328}

\bibitem{Laufer}
Henry~B. Laufer, \emph{Taut two-dimensional singularities}, Math. Ann.
  \textbf{205} (1973), 131--164. \MR{333238}

\bibitem{Laufer2}
\bysame, \emph{On minimally elliptic singularities}, Amer. J. Math. \textbf{99}
  (1977), no.~6, 1257--1295. \MR{568898}

\bibitem{Lisca-Matic}
P.~Lisca and G.~Mati\'c, \emph{Tight contact structures and {S}eiberg-{W}itten
  invariants}, Invent. Math. \textbf{129} (1997), no.~3, 509--525. \MR{1465333}

\bibitem{McDuff}
Dusa McDuff, \emph{Symplectic manifolds with contact type boundaries}, Invent.
  Math. \textbf{103} (1991), no.~3, 651--671. \MR{1091622}

\bibitem{Milnor}
John Milnor, \emph{On the {$3$}-dimensional {B}rieskorn manifolds
  {$M(p,q,r)$}}, Knots, groups, and {$3$}-manifolds ({P}apers dedicated to the
  memory of {R}. {H}. {F}ox), Ann. of Math. Stud., vol. No. 84, Princeton Univ.
  Press, Princeton, NJ, 1975, pp.~175--225. \MR{418127}

\bibitem{Mitsumatsu}
Yoshihiko Mitsumatsu, \emph{Anosov flows and non-{S}tein symplectic manifolds},
  S\'eminaire {G}aston {D}arboux de {G}\'eom\'etrie et {T}opologie
  {D}iff\'erentielle, 1994--1995 ({M}ontpellier), Univ. Montpellier II,
  Montpellier, 1995, pp.~iii, 31--42. \MR{1343776}

\bibitem{Neumann2}
Walter~D. Neumann, \emph{A calculus for plumbing applied to the topology of
  complex surface singularities and degenerating complex curves}, Trans. Amer.
  Math. Soc. \textbf{268} (1981), no.~2, 299--344. \MR{632532}

\bibitem{Neumann}
\bysame, \emph{Geometry of quasihomogeneous surface singularities},
  Singularities, {P}art 2 ({A}rcata, {C}alif., 1981), Proc. Sympos. Pure Math.,
  vol.~40, Amer. Math. Soc., Providence, RI, 1983, pp.~245--258. \MR{713253}

\bibitem{Ohta-Ono03}
Hiroshi Ohta and Kaoru Ono, \emph{Symplectic fillings of the link of simple
  elliptic singularities}, J. Reine Angew. Math. \textbf{565} (2003), 183--205.
  \MR{2024651}

\bibitem{Ohta-Ono_simple}
\bysame, \emph{Simple singularities and symplectic fillings}, J. Differential
  Geom. \textbf{69} (2005), no.~1, 1--42. \MR{2169581}

\bibitem{Ohta-Ono}
\bysame, \emph{Examples of isolated surface singularities whose links have
  infinitely many symplectic fillings}, J. Fixed Point Theory Appl. \textbf{3}
  (2008), no.~1, 51--56. \MR{2402907}

\bibitem{Raymond-Vasquez}
Frank Raymond and Alphonse~T. Vasquez, \emph{{$3$}-manifolds whose universal
  coverings are {L}ie groups}, Topology Appl. \textbf{12} (1981), no.~2,
  161--179. \MR{612013}

\bibitem{Saito}
Kyoji Saito, \emph{Einfach-elliptische {S}ingularit\"aten}, Invent. Math.
  \textbf{23} (1974), 289--325. \MR{354669}

\bibitem{Seade}
Jos\'e Seade, \emph{On the topology of isolated singularities in analytic
  spaces}, Progress in Mathematics, vol. 241, Birkh\"auser Verlag, Basel, 2006.
  \MR{2186327}

\end{thebibliography}

\end{document}